\documentclass[12pt,a4paper,psamsfonts]{amsart}
\usepackage{amssymb,amscd,amsxtra,calc}
\usepackage{cmmib57}

\newtheorem{thm}{Theorem}[section]
\newtheorem{lemma}[thm]{Lemma}
\newtheorem{proposition}[thm]{Proposition}
\newtheorem{corollary}[thm]{Corollary}
\newtheorem{claim}[thm]{Claim}

\newtheorem{conjecture}[thm]{Conjecture}

\theoremstyle{definition}
\newtheorem{definition}[thm]{Definition}
\newtheorem{remark}[thm]{Remark}

\newtheorem{construction}[thm]{Construction}

\newcommand{\pr}{\mathbb{P}}

\newcommand{\Z}{\mathbb{Z}}
\newcommand{\Q}{\mathbb{Q}}
\newcommand{\R}{\mathbb{R}}
\newcommand{\C}{\mathbb{C}}

\newcommand{\Gm}{\mathbb{G}_{\rm{m}}}
\newcommand{\NC}{\operatorname{N}_1}

\newcommand{\NE}{\operatorname{NE}}

\newcommand{\Rat}{\operatorname{RatCurves}^n}

\newcommand{\Exc}{\operatorname{Exc}}
\newcommand{\Locus}{\operatorname{Locus}}
\newcommand{\ChLocus}{\operatorname{ChLocus}}
\newcommand{\rc}{\operatorname{rc}}
\newcommand{\Chow}{\operatorname{Chow}}

\newcommand{\GL}{\operatorname{GL}}
\newcommand{\PGL}{\operatorname{PGL}}
\newcommand{\codim}{\operatorname{codim}}

\newcommand{\etale}{\text{{\'e}t}}

\newcommand{\sO}{\mathcal{O}}

\newcommand{\sE}{\mathcal{E}}

\newcommand{\sL}{\mathcal{L}}

\newcommand{\sV}{\mathcal{V}}
\newcommand{\sW}{\mathcal{W}}

\newcommand{\M}{\operatorname{M}}
\newcommand{\GM}{\operatorname{GM}}
\newcommand{\AM}{\operatorname{AM}}
\newcommand{\AGM}{\operatorname{AGM}}

\title{Around the Mukai conjecture for Fano manifolds}
\author{Kento Fujita}

\begin{document}
\maketitle
\begin{abstract}{\noindent As a generalization of the Mukai conjecture, 
we conjecture that the Fano manifolds $X$ which satisfy 
the property $\rho_X(r_X-1)\geq\dim X-1$ have very special structure, 
where $\rho_X$ is the Picard number of $X$ and $r_X$ is the index of $X$. 
In this paper, we classify those $X$ with $\rho_X\leq 3$ or $\dim X\leq 5$. }
\end{abstract}

\section{Introduction}\label{introsection}

Let $X$ be a Fano manifold, that is, a smooth projective variety such that 
the anticanonical divisor is ample. 
In this paper, we study the relationship among the Picard number $\rho_X$, 
the \emph{index} $r_X$ and the  \emph{pseudoindex} $\iota_X$. 
The definitions of index and pseudoindex are as follows: 
\begin{eqnarray*}
r_X & := & \max\{r\in\Z_{>0}\,\, |\,\, -K_X\sim rL \text{ for some Cartier divisor }L\},\\
\iota_X & := & \min\{(-K_X\cdot C)\,\, |\,\, C \text{ is a rational curve on }X\}.
\end{eqnarray*}
Clearly, $\iota_X$ is divisible by $r_X$. In particular, we have $\iota_X\geq r_X$. 

The following conjecture due to Mukai \cite{mukaiconj} is one of the most famous 
conjecture towards the relationship between 
the Picard number and the index of a Fano manifold. 

\begin{conjecture}[Mukai conjecture]\label{mukai_conj}
We have $\rho_X(r_X-1)\leq\dim X$ and equality holds if and only if 
$X\simeq(\pr^{r_X-1})^{\rho_X}$. 
\end{conjecture}

Based on the earlier work due to Wi\'sniewski \cite{wisn90}, Bonavero, 
Casagrande, Debarre and Druel \cite{BCDD} generalized Conjecture \ref{mukai_conj} 
by replacing $r_X$ by $\iota_X$. 

\begin{conjecture}[generalized Mukai conjecture]\label{gmukai_conj}
We have $\rho_X(\iota_X-1)\leq\dim X$ and equality holds if and only if 
$X\simeq(\pr^{\iota_X-1})^{\rho_X}$. 
\end{conjecture}

As in \cite{fujita}, we \emph{split} the Mukai conjecture and 
the generalized Mukai conjecture as below since we want to do certain 
inductive arguments (see Section \ref{rays_section}). 

\begin{conjecture}\label{MGM}
Let $n$ and $\rho$ be positive integers. 
\begin{enumerate}
\renewcommand{\theenumi}{\arabic{enumi}}
\renewcommand{\labelenumi}{(\theenumi)}
\item
$($Conjecture $\M^n_\rho$$)$
Let $X$ be an $n$-dimensional Fano manifold. 
If $\rho_X\geq\rho$ and $r_X\geq(n+\rho)/\rho$, then 
$X$ is isomorphic to $(\pr^{r_X-1})^{\rho}$. 
\item
$($Conjecture $\GM^n_\rho$$)$
Let $X$ be an $n$-dimensional Fano manifold. 
If $\rho_X\geq\rho$ and $\iota_X\geq(n+\rho)/\rho$, then 
$X$ is isomorphic to $(\pr^{\iota_X-1})^{\rho}$. 
\end{enumerate}
\end{conjecture}

It is obvious that the Mukai conjecture (resp.\ the generalized Mukai conjecture) 
is true if and only if Conjecture $\M^n_\rho$ (resp.\ $\GM^n_\rho$) is true 
for all positive integers $n$, $\rho$.

We conjecture that $n$-dimensional Fano manifolds $X$ with 
$\rho_X(r_X-1)\geq n-1$ (resp.\ $\rho_X(\iota_X-1)\geq n-1$) have very special 
structure. More precisely, we settle the following conjecture:

\begin{conjecture}\label{AMAGM}
Let $n$ and $\rho$ be positive integers. 
\begin{enumerate}
\renewcommand{\theenumi}{\arabic{enumi}}
\renewcommand{\labelenumi}{(\theenumi)}
\item
$($Conjecture $\AM^n_\rho$$)$
Let $X$ be an $n$-dimensional Fano manifold. 
If $\rho_X\geq\rho$ and $r_X\geq(n+\rho-1)/\rho$, then 
$X$ is isomorphic to one of the following:
\begin{enumerate}
\renewcommand{\theenumii}{\roman{enumii}}
\renewcommand{\labelenumii}{\rm{(\theenumii)}}
\item
$(\pr^{r_X-1})^{\rho}$,
\item
$\Q^{r_X}\times(\pr^{r_X-1})^{\rho-1}$,
\item
$\pr_{\pr^{r_X}}(\sO^{\oplus r_X-1}\oplus\sO(1))
\times(\pr^{r_X-1})^{\rho-2}$, 
\item
$\pr_{\pr^{r_X}}(T_{\pr^{r_X}})\times(\pr^{r_X-1})^{\rho-2}$.
\end{enumerate}
\item
$($Conjecture $\AGM^n_\rho$$)$
Let $X$ be an $n$-dimensional Fano manifold. 
If $\rho_X\geq\rho$ and $\iota_X\geq(n+\rho-1)/\rho$, then 
$X$ is isomorphic to one of the following:
\begin{enumerate}
\renewcommand{\theenumii}{\roman{enumii}}
\renewcommand{\labelenumii}{\rm{(\theenumii)}}
\item
$(\pr^{\iota_X-1})^{\rho}$,
\item
$\Q^{\iota_X}\times(\pr^{\iota_X-1})^{\rho-1}$,
\item
$\pr_{\pr^{\iota_X}}(\sO^{\oplus\iota_X-1}\oplus\sO(1))
\times(\pr^{\iota_X-1})^{\rho-2}$, 
\item
$\pr_{\pr^{\iota_X}}(T_{\pr^{\iota_X}})\times(\pr^{\iota_X-1})^{\rho-2}$,
\item
$\pr^{\iota_X}\times(\pr^{\iota_X-1})^{\rho-1}$.
\end{enumerate}
\end{enumerate}
\end{conjecture}

In particular, Conjecture $\AM^n_1$ (resp.\ Conjecture $\AGM^n_1$) asserts that 
an $n$-dimensional Fano manifold $X$ with $r_X\geq n$ (resp.\ $\iota_X\geq n$) 
will be isomorphic to either $\pr^n$ or $\Q^n$. 
The ``A" in $\AM^n_\rho$ and $\AGM^n_\rho$ stands for ``advanced". 
We note that Conjecture \ref{AMAGM} asserts in particular that the variety 
$\pr^\iota\times(\pr^{\iota-1})^{\rho-1}$ is characterized by the Fano manifold 
such that the gap between index and pseudoindex is ``largest". 

\begin{remark}\label{AGMrmk}
Clearly, Conjecture $\GM^n_\rho$ (resp.\ Conjecture $\AGM^n_\rho$) implies 
Conjecture $\M^n_\rho$ (resp.\ Conjectures $\AM^n_\rho$ and $\GM^n_\rho$). 
We also note that Conjecture $\GM^n_\rho$ is true if $n\leq 5$ (\cite{ACO}) or 
$\rho\leq 3$ (\cite{CMSB, kebekus, NO}), 
Conjecture $\AGM^n_\rho$ is true if $n\leq 3$ (\cite{isk1, isk2, sho, MoMu}), 
Conjecture $\AM^n_\rho$ is true if $n\leq 4$ (\cite{wis}) 
or $\rho\leq 2$ (\cite{KO, wisn91}), and Conjecture $\AGM^n_1$ is proved in 
\cite{miyaoka}. 
\end{remark}

In this paper, we prove Conjecture $\AM^n_\rho$ provided that $\rho\leq 3$ or
$n\leq 5$.

\begin{thm}[Main Theorem]\label{mainthm}
Conjecture $\AM^n_\rho$ is true if 
$\rho\leq 3$ or $n\leq 5$.
\end{thm}

In other words, we classify the Fano manifolds $X$ which satisfies 
the property $\rho_X(r_X-1)\geq\dim X-1$ under the condition 
$\rho_X\leq 3$ or $\dim X\leq 5$. 
We note that, as a corollary of \cite[Theorem 5.1]{novelli}, 
any $n$-dimensional Fano manifold $X$ with $\rho_X\geq 3$ and 
$r_X\geq(n+2)/3$ satisfies that either $\rho_X=3$ or $X\simeq(\pr^1)^4$. 
We rephrase Theorem \ref{mainthm} for reader's convenience. 

\begin{thm}\label{iikae}
Let $X$ be an $n$-dimensional Fano manifold. 
Suppose that $\rho_X(r_X-1)\geq n-1$. Suppose furthermore that 
either $\rho_X\leq 3$ or $n\leq 5$. 
Then $X$ is isomorphic to one of 
$(\pr^{r_X-1})^{\rho_X}$, 
$\Q^{r_X}\times(\pr^{r_X-1})^{\rho_X-1}$ $(r_X\geq 3)$, 
$\pr_{\pr^{r_X}}(\sO^{\oplus r_X-1}\oplus\sO(1))
\times(\pr^{r_X-1})^{\rho_X-2}$ or 
$\pr_{\pr^{r_X}}(T_{\pr^{r_X}})\times(\pr^{r_X-1})^{\rho_X-2}$.
\end{thm}

In order to prove Theorem \ref{mainthm}, we discuss some inductive process. 
We will prove the following proposition. 

\begin{proposition}\label{mainAGMprop}
\begin{enumerate}
\renewcommand{\theenumi}{\arabic{enumi}}
\renewcommand{\labelenumi}{\rm{(\theenumi)}}
\item\label{mainAGMprop1}
Let $n\geq 2$ and $\rho\in\{2$, $3\}$. Then Conjectures $\AGM^{n'}_{\rho-1}$ for all 
$n'\leq n-(n-1)/\rho$ imply Conjecture $\AGM^n_\rho$. 
\item\label{mainAGMprop2}
Conjecture $\AGM^n_\rho$ is true if $n\leq 5$ and $\rho\geq 2$. 
\end{enumerate}
\end{proposition}

\begin{remark}\label{miyaoka_rmk}
We do not use the deep result \cite{miyaoka} in order to prove Theorem \ref{mainthm} 
and Proposition \ref{mainAGMprop}. Obviously, if we combine 
Proposition \ref{mainAGMprop} and \cite[Theorem 0.1]{miyaoka}, then we can show 
that Conjecture $\AGM^n_\rho$ is true for $\rho\leq 3$ or $n\leq 5$. 
\end{remark}

The paper is organized as follows. 
In Sections \ref{famrat_section} and \ref{rcfib_section}, we recall definitions and some 
properties on families of rational curves and chains of rational $1$-cycles on 
Fano manifolds. The content is almost same as that of \cite[\S 2--3]{NO}. 
In Section \ref{vb_section}, we study some vector bundles on special Fano manifolds. 
This step seems crucial to consider inductive approach for proving Conjecture 
$\GM^n_\rho$ or $\AGM^n_\rho$. 
In Section \ref{rays_section}, we consider certain inductive step on $\rho$ to prove 
Conjecture $\GM^n_\rho$ or $\AGM^n_\rho$ under the additional assumption 
such that there exists a certain special extremal ray. This assumption seems 
strong, that is one of the reason why we cannot prove neither 
Conjecture $\GM^n_\rho$ nor $\AGM^n_\rho$ for general case. 
We show in Section \ref{domunsplit_section} that such an extremal ray do exists 
under the assumption that there exist many numerically independent 
dominating and unsplit families of rational curves. The argument 
is a standard technique for 
specialists; e.g., \cite[Lemma 4]{wisn91} and \cite[Theorem 1.1]{Occ}.  
We show in Section \ref{five_section} that, if $\rho\leq 3$ 
or $n\leq 5$, then there exist many numerically independent dominating and 
unsplit families of rational curves as in Section \ref{domunsplit_section}. 
In Section \ref{AM_section}, we prove Theorem \ref{mainthm} by using the techniques 
given in previous sections.

\smallskip

\noindent\textbf{Acknowledgements.}
The author thanks the referees for useful suggestions. 
The author is partially supported by JSPS Fellowships for Young Scientists.

\medskip

\noindent\textbf{Notation and terminology.}
We always work in the category of algebraic varieties 
(integral, separated and of finite type scheme) 
over the complex number field $\C$. 
For a normal projective variety $X$, we denote the normalization of the space of 
irreducible and reduced rational curves on $X$ by $\Rat(X)$ 
(see \cite[Definition II.2.11]{kollar}). 
For the theory of extremal contraction, we refer the readers to \cite{KoMo}. 
For a smooth projective variety $X$ and a $K_X$-negative extremal ray $R\subset\overline{\NE}(X)$,
\[
l(R):=\min\{(-K_X\cdot C)\mid C\text{ is a rational curve with } [C]\in R\}
\]
is called the \emph{length} $l(R)$ of $R$. 
The contraction morphism of $R$ is denoted by $\phi_R\colon X\rightarrow X_R$. 

For a morphism of varieties $f\colon X\rightarrow Y$, we define the 
\emph{exceptional locus} $\Exc(f)$ \emph{of $f$} by 
\[
\Exc(f):=\{x\in X\mid f \text{ is not an isomorphism around }x\}.
\]

For a complete variety $X$, an invertible sheaf $\sL$ 
on $X$ and for a nonnegative integer $i$, 
we denote the dimension of the $\C$-vector space $H^i(X,\sL)$ by $h^i(X,\sL)$. 
We also define $h^i(X,L)$ as $h^i(X,\sO_X(L))$
for a Cartier divisor $L$ on $X$. 

For a complete variety $X$, the Picard number of $X$ is denoted by $\rho_X$. 
For a complete variety $X$ and a closed subvariety $Y\subset X$, we denote the image 
of the homomorphism $\NC(Y)\rightarrow\NC(X)$ by $\NC(Y, X)$. 

For an algebraic scheme $X$ and a locally free sheaf of finite rank $\sE$ on $X$, 
let $\pr_X(\sE)$ be the projectivization of $\sE$ in the sense of Grothendieck 
and $\sO_\pr(1)$ be the tautological invertible sheaf. We usually denote 
the projection by $p\colon\pr_X(\sE)\rightarrow X$. 
We use the terms ``vector bundle" and ``locally free sheaf of finite rank" 
interchangeably. 
For a smooth projective variety $X$, let $T_X$ be the tangent bundle of $X$. 

The symbol $\Q^n$ means a smooth hyperquadric in $\pr^{n+1}$ for $n\geq 2$. 
We write $\sO_{\Q^n}(1)$ as the invertible sheaf 
which is the restriction of $\sO_{\pr^{n+1}}(1)$ under the natural embedding. 
We sometimes write $\sO(m)$ instead of $\sO_{\Q^n}(m)$ on $\Q^n$ 
(or $\sO_{\pr^n}(m)$ on $\pr^n$) for simplicity.

\medskip

\section{Families of rational curves}\label{famrat_section}

\smallskip

We recall the definition and properties of a family of rational curves for a fixed 
smooth projective variety. For detail, see \cite{kollar}. 

\begin{definition}\label{no_d1}
Let $X$ be a smooth projective variety. 
We define a \emph{family of rational curves} on $X$ 
to be an irreducible component $V\subset\Rat(X)$ with the induced 
universal family. 
For any $x\in X$, let $V_x$ be the subvariety of $V$ 
parameterizing rational curves passing through $x$. 
We define $\Locus(V)$ (resp.\ $\Locus(V_x)$) to be the union of rational curves 
parametrized by $V$ (resp.\ $V_x$). 
For a Cartier divisor $L$ on $X$, 
the intersection number $(L\cdot C)$ for any rational curve 
$C$ whose class belongs to $V$ is denoted by $(L\cdot V)$. 
We also denote by $[V]\in\NC(X)$ the numerical class of any rational curve among 
those parametrized by $V$. 

For a family $V$ of rational curves on $X$,
the family $V$ is said to be \emph{dominating} if $\overline{\Locus(V)}=X$, 
\emph{unsplit} if $V$ is projective, 
and \emph{locally unsplit} if $V_x$ is projective for a general $x\in\Locus(V)$.
If $V$ is a locally unsplit family, then $(-K_X\cdot V)\leq\dim X+1$ holds 
by \cite[Theorem 4]{mori}. 

If $X$ is a Fano manifold, then $X$ admits a dominating family of rational curves 
by \cite[Theorem 6]{mori}. If a dominating family $V$ of rational curves on $X$ satisfies 
that the intersection number $(-K_X\cdot V)$ is minimal among such $V$, then the 
family $V$ is called by a \emph{minimal dominating family} of $X$. 
We note that a minimal dominating family is locally unsplit. 
\end{definition}

\begin{definition}\label{no_d2}
Let $X$ be a Fano manifold, $U\subset X$ be an open subvariety and 
$\pi\colon U\rightarrow Z$ be a proper surjective morphism to a quasiprojective 
variety $Z$ of positive dimension. A family $V$ of rational curves on $X$ is a 
\emph{horizontal dominating family with respect to $\pi$} if $\Locus(V)$ dominates 
$Z$ and curves parametrized by $V$ are not contracted by $\pi$. 
We know that such a family always exists by \cite[Theorem 2.1]{KoMiMo}. 
A horizontal dominating family $V$ of rational curves on $X$ with respect to $\pi$ 
is  called a \emph{minimal horizontal dominating family} with respect to $\pi$ 
if the intersection number $(-K_X\cdot V)$ is minimal among such $V$. 
We note that a minimal horizontal dominating family is locally unsplit. 
\end{definition}

\begin{definition}\label{no_d3}
Let $X$ be a smooth projective variety. 
We define a \emph{Chow family $\sW$ of rational 
$1$-cycles} on $X$ 
to an irreducible component of the Chow variety $\Chow(X)$ of $X$ 
parameterizing rational and connected $1$-cycles. 
We define $\Locus(\sW)$ to be the union of the supports of $1$-cycles 
parametrized by $\sW$. We say that $\sW$ is a \emph{covering family} if 
$\Locus(\sW)=X$. 

For a family $V$ of rational curves on $X$, the closure of the image of $V$ in 
$\Chow(X)$ is denoted by $\sV$ and called the \emph{Chow family associated to $V$}. 
If $V$ is unsplit, then $V$ is the normalization of $\sV$ by \cite[II.2.11]{kollar}. 

For a family $V$ of rational curves on $X$, we say that $V$ (and also $\sV$) is 
\emph{quasi-unsplit} if any component of any reducible cycle parametrized by $\sV$ 
is numerically proportional to the class of curves parametrized by $V$. 

If families $V^1,\dots,V^k$ of rational curves on $X$ satisfy that 
the dimension of the vector space $\sum_{i=1}^k\R[V^i]$ in $\NC(X)$ is equal to $k$, 
then we say that $V^1,\dots,V^k$ are \emph{numerically independent}. 
\end{definition}

\begin{definition}\label{no_d5}
Let $X$ be a smooth projective variety, $V^1,\dots,V^k$ be families of rational 
curves on $X$ and $Y\subset X$ be a closed subvariety. 
We define 
\[
{\Locus(V^1)}_Y:={\bigcup}_{[C]\in V^1;Y\cap C\ne\emptyset}C,
\]
and we inductively define $\Locus(V^1,\dots,V^k)_Y:=
\Locus(V^k)_{\Locus(V^1,\dots,V^{k-1})_Y}$. 
Analogously, we define $\Locus(\sW^1,\dots,\sW^k)_Y$ for Chow families 
$\sW^1,\dots,\sW^k$ of rational $1$-cycles. 

For any point $x\in X$, 
we define $\Locus(V^1,\dots,V^k)_x:=\Locus(V^1,\dots,V^k)_{\{x\}}$. 
\end{definition}

The following assertions are well-known. We omit the proof. 

\begin{proposition}[{see \cite[Corollary IV.2.6]{kollar}}]\label{no_p1}
Let $X$ be a smooth projective variety, $V$ be a family of rational curves on $X$ 
and $x\in\Locus(V)$ be a $($closed$)$ point such that $V_x$ is projective. 
Then the dimension of any irreducible component of $\Locus(V_x)$ is 
bigger than or equal to 
\[
\dim X-\dim\Locus(V)+(-K_X\cdot V)-1.
\]
\end{proposition}

\begin{proposition}[{see \cite[Proposition 2]{NO}}]\label{no_p2}
Let $V$ be a dominating and locally unsplit family of rational curves on a smooth 
projective variety $X$ and $\sV$ be the associated Chow family. Assume that 
$\dim\Locus(V_x)\geq s$ for a general $x\in X$ and some integer $s$, then for any 
$x\in X$ every irreducible component of $\Locus(\sV)_x$ has dimension $\geq s$. 
\end{proposition}

\begin{lemma}[{see \cite[Lemma 5.4]{ACO}}]\label{no_l1}
Let $X$ be a smooth projective variety, $Y\subset X$ be a closed subvariety and 
$V^1,\dots,V^k$ be numerically independent unsplit families of rational curves on $X$. 
Assume that $(\sum_{i=1}^k\R[V^k])\cap\NC(Y, X)=0$ and 
$\Locus(V^1,\dots,V^k)_Y\neq\emptyset$. Then we have 
\[
\dim\Locus(V^1,\dots,V^k)_Y\geq\dim Y+\sum_{i=1}^k\left((-K_X\cdot V^i)-1\right).
\]
\end{lemma}

\begin{lemma}[{\cite[Lemma 4.1]{ACO}}]\label{no_l2}
Let $X$ be a smooth projective variety, $Y\subset X$ be a closed subvariety and 
$\sW$ be a Chow family of rational $1$-cycles on $X$. 
Then any curve in $\Locus(\sW)_Y$ 
is numerically equivalent to a linear combination of rational coefficients of curves 
in $Y$ and of irreducible components of cycles parametrized by $\sW$ which meet $Y$. 
\end{lemma}

\begin{lemma}[{\cite[Corollary 1]{NO}}]\label{no_l3}
Let $X$ be a smooth projective variety, $V^1$ be a locally unsplit family 
of rational curves on $X$ and 
$V^2,\dots,V^k$ be unsplit families of rational curves on $X$. 
Then for a general $x\in\Locus(V^1)$, we have the following results. 
\begin{enumerate}
\renewcommand{\theenumi}{\alph{enumi}}
\renewcommand{\labelenumi}{\rm{(\theenumi)}}
\item\label{no_l31}
$\NC(\Locus(V^1)_x, X)=\R[V^1]$ holds.  
\item\label{no_l32}
If $\Locus(V^1,\dots,V^k)_x\neq\emptyset$, then $\NC(\Locus(V^1,\dots,V^k)_x, X)
=\sum_{i=1}^k\R[V^i]$ holds. 
\end{enumerate}
\end{lemma}

\medskip

\section{Rationally connected fibrations}\label{rcfib_section}

\smallskip

In this section, we recall the theory of rationally connected fibrations. 
For detail, see \cite{kollar} and \cite[\S 3]{NO}.

\begin{definition}[{see \cite[IV.4]{kollar}, \cite[\S 3]{ACO}}]\label{no_d7}
Let $X$ be a smooth projective variety, $Y\subset X$ be a closed subvariety, 
$m$ be a positive integer and 
$\sW^1,\dots,\sW^k$ be Chow families of rational $1$-cycles on $X$. 
We define $\ChLocus(\sW^1,\dots,\sW^k)_Y$ 
to be the set of points $y\in X$ such that 
there exist cycles $\Gamma_1,\dots,\Gamma_m$ with the following properties: 
\begin{enumerate}
\renewcommand{\theenumi}{\alph{enumi}}
\renewcommand{\labelenumi}{\rm{(\theenumi)}}
\item\label{no_d71}
The cycle $\Gamma_i$ belongs to one of the families $\sW^1,\dots,\sW^k$ 
for any $1\leq i\leq m$, 
\item\label{no_d72}
$\Gamma_i\cap\Gamma_{i+1}\neq\emptyset$ for any $1\leq i\leq m-1$,  
\item\label{no_d73}
$\Gamma_1\cap Y\neq\emptyset$ and $y\in\Gamma_m$. 
\end{enumerate}
For a point $x\in X$, we define $\ChLocus_m(\sW^1,\dots,\sW^k)_x
:=\ChLocus_m(\sW^1,\dots,\sW^k)_{\{x\}}$.

We say that two points $x$, $y\in X$ are \emph{$\rc(\sW^1,\dots,\sW^k)$-equivalent} 
if there exists $m\in\Z_{>0}$ such that $y\in\ChLocus_m(\sW^1,\dots,\sW^k)_x$. 

We say that $X$ is \emph{$\rc(\sW^1,\dots,\sW^k)$-connected} if 
$X=\ChLocus_m(\sW^1,\dots,\sW^k)_x$ holds 
for some $m$ and for some (hence any) $x\in X$. 
\end{definition}

\begin{thm}[{\cite[Theorem IV.4.16]{kollar}}]\label{no_t1}
Let $X$ be a smooth projective variety and 
$\sW^1,\dots,\sW^k$ be Chow families of rational $1$-cycles on $X$. 
Then there exists an open subvariety $X^0\subset X$ and a proper surjective morphism 
with connected fibers $\pi\colon X^0\rightarrow Z^0$ to a quasiprojective variety $Z^0$ 
such that the following holds:
\begin{itemize}
\item
The equivalence relation obtained by the $\rc(\sW^1,\dots,\sW^k)$-equivalence
restricts to an equivalence relation on $X^0$. 
\item
$\pi^{-1}(z)$ coincides with an $\rc(\sW^1,\dots,\sW^k)$-equivalence class for any 
$z\in Z^0$. 
\item
For any $z\in Z^0$ and $x$, $y\in\pi^{-1}(z)$, we have $y\in\ChLocus_m
(\sW^1,\dots,\sW^k)_x$ for some $m\leq 2^{\dim X-\dim Z^0}-1$. 
\end{itemize}
We call this morphism the \emph{$\rc(\sW^1,\dots,\sW^k)$-fibration} and often write 
$\pi\colon X\dashrightarrow Z$ for simplicity $($where $Z$ is a projective variety$)$. 
\end{thm}

\begin{proposition}[{see \cite[Corollary 4.4]{ACO}}]\label{no_p3}
Let $X$ be a smooth projective variety and 
$\sW^1,\dots,\sW^k$ be Chow families of rational $1$-cycles on $X$. 
If $X$ is $\rc(\sW^1,\dots,\sW^k)$-connected, then $\NC(X)$ is spanned by the 
classes of irreducible components of cycles in $\sW^1,\dots,\sW^k$. 
In particular, if $\sW^i$ is the Chow family associated to some quasi-unsplit family 
$W^i$ of rational curves on $X$ for any $1\leq i\leq k$, then $\rho_X\leq k$ 
and equality holds if and only if $W^1,\dots,W^k$ are numerically independent. 
\end{proposition}

\begin{thm}[{cf.\ \cite[Theorem 2]{NO}}]\label{no_t2}
Let $X$ be a Fano manifold and $V$ be a dominating and locally unsplit 
family of rational curves on $X$. 
Assume that $X$ is $\rc(\sV)$-connected and $(-K_X\cdot V)<3\iota_X$ holds. 
\begin{enumerate}
\renewcommand{\theenumi}{\arabic{enumi}}
\renewcommand{\labelenumi}{\rm{(\theenumi)}}
\item\label{no_t21}
If $V$ is a minimal dominating family and $(-K_X\cdot V)>\dim X+1-\iota_X$, 
then $\rho_X=1$. 
\item\label{no_t22}
If $(-K_X\cdot V)>\dim X+1-\iota_X$, 
then $\rho_X\leq 2$. 
\item\label{no_t23}
If $(-K_X\cdot V)\geq\dim X+1-\iota_X$ and $\iota_X\geq 2$, 
then $\rho_X\leq 3$. 
\end{enumerate}
\end{thm}

\begin{proof}
The proof is almost same as that of \cite[Theorem 2]{NO}. 

Fix a general point $x\in X$. There exists $m\in\Z_{>0}$ such that 
$X=\ChLocus_m(\sV)_x$ since $X$ is $\rc(\sV)$-connected.
Since $(-K_X\cdot V)<3\iota_X$, any reducible cycle $\Gamma$ of $\sV$ has only 
two irreducible components. Hence either both of them are numerically proportional 
to $[V]\in\NC(X)$ or neither of them is numerically proportional to $[V]\in\NC(X)$. 

If any irreducible component of an $m$-chain $\Gamma_1\cup\dots\cup\Gamma_m$ 
which satisfies 
\begin{enumerate}
\renewcommand{\theenumi}{\roman{enumi}}
\renewcommand{\labelenumi}{\rm{(\theenumi)}}
\item\label{no_t201}
$x\in\Gamma_1$ and 
\item\label{no_t202}
$\Gamma_i\cap\Gamma_{i+1}\neq\emptyset$ for any $1\leq i\leq m-1$
\end{enumerate}
is numerically proportional to $[V]\in\NC(X)$, 
then $\rho_X=1$ by Proposition \ref{no_p3}. 

We can assume that there exists an $m$-chain $\Gamma_1\cup\dots\cup\Gamma_m$ 
which satisfies the properties \eqref{no_t201}, \eqref{no_t202} and there exists 
an integer $1\leq j\leq m$ such that the irreducible components $\Gamma_j^1$ 
and $\Gamma_j^2$ of $\Gamma_j$ are not numerically proportional to $[V]\in\NC(X)$. 
Let $1\leq j_0\leq m$ be the minimum integer for which such a chain exists. 
We have $j_0\geq 2$ since $x\in X$ is general. 
If $j_0=2$ then set $x_1:=x$, otherwise let $x_1\in X$ be a point in 
$\Gamma_{j_0-2}\cap\Gamma_{j_0-1}$. Take an irreducible component $Y$ of 
$\Locus(V_{x_1})$ which meets $\Gamma_{j_0}$. We can assume 
that $\Gamma_{j_0}^1\cap Y\neq\emptyset$. 
We know that $\NC(Y, X)=\R[V]$ by Lemma \ref{no_l2} and the minimality of $j_0$. 
Take a family $W$ of rational curves on $X$ such that the class of $\Gamma_{j_0}^1$ 
is in $W$. Then $W$ is unsplit by the property $(-K_X\cdot V)<3\iota_X$. 
By Lemma \ref{no_l1}, Propositions \ref{no_p1} and \ref{no_p2}, 
we have $\dim\Locus(W)_Y\geq\dim Y+(-K_X\cdot W)-1\geq (-K_X\cdot V)+\iota_X-2$. 

\eqref{no_t21}
We have $\Locus(W)_Y=X$ since $(-K_X\cdot V)>\dim X+1-\iota_X$. 
In particular, $W$ is a dominating family. However, this leads to a contradiction 
since $V$ is a minimal dominating family and $(-K_X\cdot V)>(-K_X\cdot W)$. 
Thus $\rho_X=1$. 

\eqref{no_t22}
We have $\Locus(W)_Y=X$ by the same reason. 
We know that $\NC(\Locus(W)_Y, X)=\R[V]+\R[W]$ by Lemma \ref{no_l2}. 
Thus $\rho_X\leq 2$. 

\eqref{no_t23}
We have $\Locus(W)_Y$ is a divisor or equal to $X$ and 
$\NC(\Locus(W)_Y, X)=\R[V]+\R[W]$ by the same reason. 
If $\Locus(W)_Y$ is equal to $X$, then $\rho_X\leq 2$. 
If $\Locus(W)_Y$ is a divisor, then $\rho_X\leq 3$ by \cite[Theorem 1.2]{casagrande}. 
\end{proof}

We recall the following argument due to Novelli and Occhetta.

\begin{construction}[{\cite[Construction 1]{NO}}]\label{const}
Let $X$ be a Fano manifold. 
Take a minimal dominating family $V^1$ of rational curves on $X$. 
If $X$ is not $\rc(\sV^1)$-connected, take a minimal horizontal dominating 
family $V^2$ of rational curves on $X$ with respect to the $\rc(\sV^1)$-fibration 
$\pi^1\colon X\dashrightarrow Z^1$. 
If $X$ is not $\rc(\sV^1, \sV^2)$-connected, take a minimal horizontal dominating 
family $V^3$ of rational curves on $X$ with respect to the $\rc(\sV^1, \sV^2)$-fibration 
$\pi^2\colon X\dashrightarrow Z^2$, and so on. Since $\dim Z^{i+1}<\dim Z^i$, 
for some integer $k$ we have that $X$ is $\rc(\sV^1,\dots,\sV^k)$-connected. 
We note that the families $V^1,\dots,V^k$ are numerically independent by construction. 
\end{construction}

\begin{lemma}[{see \cite[Lemma 4]{NO}}]\label{no_l4}
Let $X$ be a Fano manifold with $\iota_X\geq 2$ and $V^1,\dots,V^k$ be families of 
rational curves as in Construction \ref{const}. Then we have 
\begin{eqnarray*}
\dim X & \geq & \sum_{i=1}^k\dim\left((\pi^i)^{-1}\left(\pi^i(x_i)\right)\right)\geq
\sum_{i=1}^k\dim\Locus(V^i)_{x_i}\\
 & \geq & \sum_{i=1}^k\left(\dim X-\dim\Locus(V^i)+(-K_X\cdot V^i)-1\right)\geq
\sum_{i=1}^k\left((-K_X\cdot V^i)-1\right)
\end{eqnarray*}
for any general $x_i\in\Locus(V^i)$. 
\end{lemma}

\begin{lemma}[{\cite[Lemma 4.5]{novelli}}]\label{novlem}
Let $X$ be a Fano manifold with $\iota_X\geq 2$ and $V^1,\dots,V^k$ be families of 
rational curves as in Construction \ref{const}. Assume that at least one of these 
families, say $V^j$, is non-unsplit. Then $k(\iota_X-1)\leq\dim X-\iota_X$. Moreover, 
\begin{enumerate}
\renewcommand{\theenumi}{\alph{enumi}}
\renewcommand{\labelenumi}{\rm{(\theenumi)}}
\item\label{novlem1}
if $j=(\dim X-\iota_X)/(\iota_X-1)$, then $j=k$ 
and $\rho_X(\iota_X-1)=\dim X-\iota_X$; 
\item\label{novlem1}
if $j=(\dim X-\iota_X-1)/(\iota_X-1)$, then $j=k$ 
and either $\rho_X(\iota_X-1)=\dim X-\iota_X-1$, 
or $\iota_X=2$ and $\rho_X=\dim X-2$. 
\end{enumerate}
\end{lemma}

\medskip

\section{Special vector bundles}\label{vb_section}

\smallskip

In this section, we consider vector bundles on some special Fano manifolds 
whose projectivizations are also Fano manifolds with large pseudoindices. 

\begin{definition}\label{pnbdle_dfn}
A morphism $f\colon X\rightarrow Y$ is called a \emph{$\pr^m$-fibration} if 
$f$ is a proper and smooth morphism such that the scheme theoretic fiber of $f$ is 
isomorphic to $\pr^m$ for any (closed) point in $Y$.
\end{definition}

The following lemma in \cite{BCDD} is fundamental. 

\begin{lemma}[{\cite[Lemme 2.5 (a)]{BCDD}}]\label{pnbdle_lemma}
Let $f\colon X\rightarrow Y$ be a $\pr^m$-fibration between smooth projective 
varieties. 
If $X$ is a Fano manifold, then $Y$ is also a Fano manifold and $\iota_Y\geq\iota_X$ holds. 
\end{lemma}

We give a sufficient condition that a given $\pr^m$-fibration is isomorphic to 
a projective space bundle. 

\begin{proposition}\label{brauer_prop}
Let $f\colon X\rightarrow Y$ be a $\pr^m$-fibration 
between smooth projective varieties. 
If $Y$ is a rational variety, i.e., birational to a projective space, then $f$ is 
a projective space bundle.  
More precisely, there exists a locally free sheaf $\sE$ of rank $m+1$ 
on $Y$ such that $X$ is isomorphic to $\pr_Y(\sE)$ over $Y$.
\end{proposition}

\begin{proof}
Since $Y$ is a smooth projective rational variety, the cohomological Brauer group 
$H^2_{\etale}(Y, \Gm)$ of $Y$ is equal to zero (see for example \cite[\S 5]{CS}). 
Thus the homomorphism 
$H^1_\etale(Y, \GL_{m+1})\rightarrow H^1_\etale(Y, \PGL_{m+1})$ 
is surjective. 
\end{proof}

We introduce the notion of minimal horizontal curves of projective space bundles 
over rational curves. The idea focusing on those curves 
has been already obtained in \cite[\S 2]{BCDD}. 

\begin{definition}\label{hor_section}
Let $Y$ be a smooth projective variety, let $\sE$ be a locally free sheaf on $Y$ of 
rank $m+1$ and let $X:=\pr_Y(\sE)$ with the projection $p\colon X\rightarrow Y$. 
Let $C\subset Y$ be a rational curve with the normalization morphism 
$\nu\colon \pr^1\rightarrow C\hookrightarrow Y$. 
Consider the fiber product 
\[
\begin{CD}
\pr_{\pr^1}(\nu^*\sE)       @>{\nu'}>>  X                        \\
@V{p'}VV                                        @V{p}VV               \\
\pr^1                             @>>{\nu}>  Y.
\end{CD}
\]
There exists an isomorphism 
\[
\nu^*\sE\simeq\bigoplus_{0\leq i\leq m}\sO_{\pr^1}(a_i)
\]
with $a_0\leq\dots\leq a_m$. 
Let $C'\subset\pr_{\pr^1}(\nu^*\sE)$ be the section of $p'$ corresponds to 
the canonical projection 
\[
\nu^*\sE\simeq\bigoplus_{0\leq i\leq m}\sO_{\pr^1}(a_i)\twoheadrightarrow\sO_{\pr^1}(a_0)
\]
and let $C^{p, 0}\subset X$ be the image of $C'$ in $X$. We call this $C^{p,0 }$ 
a \emph{minimal horizontal curve of $p$ over $C$}. 
The choice of $C^{p, 0}$ is not unique in general. However, we have 
\begin{eqnarray}\label{hor_eqn}
(-K_X\cdot C^{p, 0})
=(-K_Y\cdot C)-\sum_{i=1}^m(a_i-a_0)
\end{eqnarray}
since $(-K_X\cdot C^{p, 0})=
(p^*(\sO_Y(-K_Y)\otimes(\det\sE)^\vee)\otimes\sO_\pr(m+1)\cdot C^{p, 0})$, 
$\deg(\det(\nu^*\sE))=\sum_{i=0}^ma_i$ 
and $(\sO_\pr(1)\cdot C')_{\pr_{\pr^1}(\nu^*\sE)}=a_0$. 
This value does not depend on the choice of $C^{p, 0}$. 
\end{definition}

\begin{lemma}\label{vb_lift}
Let $Z$ be a smooth projective variety and $Y:=\pr^m\times Z$ $($we allow the case 
$Z$ is a point$)$. We write the projections $p_1\colon Y\rightarrow\pr^m$ and 
$p_2\colon Y\rightarrow Z$. 
Let $\sE$ be a locally free sheaf on $Y$ of rank $m+1$ and $X:=\pr_Y(\sE)$ 
with the projection $p\colon X\rightarrow Y$. Assume that $X$ is a Fano manifold 
with $\iota_X\geq m+1$. Then there exist an integer $a$ and a locally free sheaf 
$\sE_Z$ on $Z$ of rank $m+1$ such that 
$\sE\simeq p_1^*\sO_{\pr^m}(a)\otimes p_2^*\sE_Z$ holds. Moreover, 
$X_Z:=\pr_Z(\sE_Z)$ satisfies that $X\simeq X_Z\times\pr^m$. In particular, 
$X_Z$ is also a Fano manifold with $\iota_{X_Z}\geq m+1$. 
\end{lemma}

\begin{proof}
Pick any (closed) point $z\in Z$ and any line $l\subset p_2^{-1}(z)(=\pr^m)\subset Y$. 
Then we have $\sE|_l\simeq\bigoplus_{0\leq i\leq m}\sO_{\pr^1}(a)$ for some $a\in\Z$ 
by the equation \eqref{hor_eqn} and the properties $(-K_Y\cdot l)=m+1$ and 
$(-K_X\cdot l^{p, 0})\geq m+1$. The integer $a$ does not depend on the choices of $z$ and 
$l$ since the value $(\det\sE\cdot l)=(m+1)a$ is independent of the 
choices of $z$ and $l$. 
Thus $\sE':=\sE\otimes p_1^*\sO_{\pr^m}(-a)$ satisfies that 
$\sE'|_l\simeq\sO_{\pr^1}^{\oplus m+1}$ for any (closed) point $z\in Z$ and 
any line $l\subset p_2^{-1}(z)\subset Y$. 
Thus $\sE'|_{p_2^{-1}(z)}\simeq\sO_{\pr^m}^{\oplus m+1}$ by \cite[Proposition (1.2)]{AW}. 
We have $h^0({p_2^{-1}(z)}, \sE'|_{p_2^{-1}(z)})=m+1$ and 
$h^1({p_2^{-1}(z)}, \sE'|_{p_2^{-1}(z)})=0$. 
Hence $\sE_Z:=(p_2)_*\sE'$ is a locally free sheaf on $Z$ of rank $m+1$ and 
$p_2^*\sE_Z\simeq\sE'$ holds by the cohomology and base change theorem. Therefore 
we have $\sE\simeq p_1^*\sO_{\pr^m}(a)\otimes p_2^*\sE_Z$. 
The remaining assertions are trivial. 
\end{proof}

\begin{corollary}\label{vb_gm}
Let $Y:=(\pr^m)^k$ for some $m$, $k\geq 1$, let $\sE$ be a locally free sheaf on $Y$ 
of rank $m+1$ and let $X:=\pr_Y(\sE)$ with the projection $p\colon X\rightarrow Y$. 
If $X$ is a Fano manifold with $\iota_X\geq m+1$, then $X$ is isomorphic to 
$(\pr^m)^{k+1}$. 
\end{corollary}

\begin{proof}
It is obvious from Lemma \ref{vb_lift} by using induction on $k$. 
\end{proof}

\begin{corollary}\label{vb_agm}
Fix $m$, $k\geq 1$. Let $Y$ be a smooth projective variety, let $\sE$ be a locally free 
sheaf on $Y$ of rank $m+1$ and let $X:=\pr_Y(\sE)$ with the projection 
$p\colon X\rightarrow Y$. Assume that $X$ is a Fano manifold with $\iota_X\geq m+1$. 
\begin{enumerate}
\renewcommand{\theenumi}{\alph{enumi}}
\renewcommand{\labelenumi}{\rm{(\theenumi)}}
\item\label{vb_qp}
If $Y=\Q^{m+1}\times(\pr^m)^{k-1}$, then $X\simeq Y\times\pr^m$. 
\item\label{vb_bp}
If $Y=\pr_{\pr^{m+1}}(\sO^{\oplus m}\oplus\sO(1))
\times(\pr^m)^{k-1}$, 
then $X\simeq Y\times\pr^m$. 
\item\label{vb_tp}
If $Y=\pr_{\pr^{m+1}}(T_{\pr^{m+1}})\times(\pr^m)^{k-1}$,
then $X\simeq Y\times\pr^m$. 
\item\label{vb_pp}
If $Y=\pr^{m+1}\times(\pr^m)^{k-1}$, then $X$ is isomorphic to one of the following: 
\begin{enumerate}
\renewcommand{\theenumii}{\roman{enumii}}
\renewcommand{\labelenumii}{\rm{(\theenumii)}}
\item\label{vb_prod}
$Y\times\pr^m$, 
\item\label{vb_pbp}
$\pr_{\pr^{m+1}}(\sO^{\oplus m}\oplus\sO(1))\times(\pr^m)^{k-1}$, 
\item\label{vb_ptp}
$\pr_{\pr^{m+1}}(T_{\pr^{m+1}})\times(\pr^m)^{k-1}$. 
\end{enumerate}
\end{enumerate}
\end{corollary}

\begin{proof}
We can assume $k=1$ by Lemma \ref{vb_lift}. 

\eqref{vb_pp} 
Take any line $l\subset Y=\pr^{m+1}$. 
Then the locally free sheaf $\sE|_l$ is either isomorphic to 
\begin{enumerate}
\renewcommand{\theenumi}{\arabic{enumi}}
\renewcommand{\labelenumi}{\rm{(\theenumi)}}
\item\label{vb_triv}
$\sO_{\pr^1}(a)^{\oplus m+1}$ or 
\item\label{vb_sato}
$\sO_{\pr^1}(a)^{\oplus m}\oplus\sO_{\pr^1}(a+1)$
\end{enumerate}
for some $a\in\Z$
by the equation \eqref{hor_eqn} and the properties $(-K_Y\cdot l)=m+2$ and 
$(-K_X\cdot l^{p, 0})\geq m+1$. Moreover, the possibility \eqref{vb_triv} or 
\eqref{vb_sato} and the integer $a$ do not depend on the choice of $l$. 
If the case \eqref{vb_triv} occurs, then $\sE\otimes\sO_{\pr^{m+1}}(-a)\simeq
\sO_{\pr^{m+1}}^{\oplus m+1}$ by \cite[Proposition (1.2)]{AW}. Thus $X$ is isomorphic 
to $\pr^{m+1}\times\pr^m$. 
If the case \eqref{vb_sato} occurs, then 
$\sE$ is isomorphic to either $\sO_{\pr^{m+1}}(a)^{\oplus m}\oplus\sO_{\pr^{m+1}}(a+1)$ 
or $T_{\pr^{m+1}}\otimes\sO_{\pr^{m+1}}(a-1)$ 
by \cite[Main Theorem 2) (ii)]{sato}. Thus $X$ is isomorphic to either 
$\pr_{\pr^{m+1}}(\sO^{\oplus m}\oplus\sO(1))$ or $\pr_{\pr^{m+1}}(T_{\pr^{m+1}})$. 

\eqref{vb_qp} 
If $m=1$, then the assertion is true by Corollary \ref{vb_gm}. 
We can assume that $m\geq 2$. 
Take any line $l\subset Y=\Q^{m+1}$. 
Then we have $\sE|_l\simeq\sO_{\pr^1}(a)^{\oplus m+1}$ for some $a\in\Z$ 
by the equation \eqref{hor_eqn} and the properties $(-K_Y\cdot l)=m+1$ and 
$(-K_X\cdot l^{p, 0})\geq m+1$. Moreover, the integer $a$ does not 
depend on the choice of $l$. 
Then $\sE\otimes\sO_{\Q^{m+1}}(-a)\simeq
\sO_{\Q^{m+1}}^{\oplus m+1}$ by \cite[Proposition (1.2)]{AW}. 
Thus $X$ is isomorphic to $\Q^{m+1}\times\pr^m$. 

\eqref{vb_bp} 
Let $p'\colon Y=\pr_{\pr^{m+1}}(\sO^{\oplus m}\oplus\sO(1))\rightarrow\pr^{m+1}$ 
be the projection and $q\colon Y\rightarrow\pr^{2m+1}$ be the blowing 
up along an $(m-1)$-dimensional linear subspace. 
Take any (closed) point $z\in\pr^{m+1}$ and any line $l\subset (p')^{-1}(z)(\simeq\pr^m)
\subset Y$. 
Then we have $\sE|_l\simeq\sO_{\pr^1}(a)^{\oplus m+1}$ for some $a\in\Z$ 
by the equation \eqref{hor_eqn} and the properties $(-K_Y\cdot l)=m+1$ and 
$(-K_X\cdot l^{p, 0})\geq m+1$. Moreover, the integer $a$ do not 
depend on the choices of $z$ and $l$. 
Then $\sE':=\sE\otimes q^*\sO_{\pr^{2m+1}}(-a)$ satisfies that 
$\sE'|_{(p')^{-1}(z)}\simeq\sO_{\pr^m}^{\oplus m+1}$ 
for any (closed) point $z\in\pr^{m+1}$. 
Thus $\sE_1:=p'_*\sE'$ is a locally free sheaf on $\pr^{m+1}$ of rank $m+1$ and 
$(p')^*\sE_1\simeq\sE'$ holds by the cohomology and base change theorem. 
Hence $\sE_1$ is isomorphic to one of the following 
\begin{enumerate}
\renewcommand{\theenumi}{\arabic{enumi}}
\renewcommand{\labelenumi}{\rm{(\theenumi)}}
\item\label{vb_bp_triv}
$\sO_{\pr^{m+1}}(b)^{\oplus m+1}$, 
\item\label{vb_bp_sato1}
$\sO_{\pr^{m+1}}(b)^{\oplus m}\oplus\sO_{\pr^{m+1}}(b+1)$ or 
\item\label{vb_bp_sato2}
$T_{\pr^{m+1}}\otimes\sO_{\pr^{m+1}}(b-1)$
\end{enumerate}
for some $b\in\Z$ by \eqref{vb_pp}. 
Take a line $l'$ in a nontrivial fiber $(\simeq\pr^{m+1})$ of $q$. 
Then we have $\sE|_{l'}\simeq\sO_{\pr^1}(a')^{\oplus m+1}$ for some $a'\in\Z$ 
by the equation \eqref{hor_eqn} and the properties $(-K_Y\cdot l')=m+1$ and 
$(-K_X\cdot (l')^{p, 0})\geq m+1$. 
Thus we have $(m+1)a'=(\det\sE\cdot l')=(\det\sE_1\cdot p_*l')$. 
If $\sE_1$ is isomorphic to either of type \eqref{vb_bp_sato1} or \eqref{vb_bp_sato2}, 
then $(\det\sE_1\cdot p_*l')=(m+1)b+1$. This leads to a contradiction. Hence 
$\sE_1\simeq\sO_{\pr^{m+1}}(b)^{\oplus m+1}$. 
In particular $X$ is isomorphic to $\pr_{\pr^{m+1}}(\sO^{\oplus m}\oplus\sO(1))
\times\pr^m$. 

\eqref{vb_tp} 
Let $p'\colon Y=\pr_{\pr^{m+1}}(T_{\pr^{m+1}})\rightarrow\pr^{m+1}$ 
be the projection and $q\colon Y\rightarrow\pr^{m+1}$ be the other contraction 
morphism. 
Take any (closed) point $z\in\pr^{m+1}$ and any line $l\subset (p')^{-1}(z)(\simeq\pr^m)
\subset Y$. 
Then we have $\sE|_l\simeq\sO_{\pr^1}(a)^{\oplus m+1}$ for some $a\in\Z$ 
by the equation \eqref{hor_eqn} and the properties $(-K_Y\cdot l)=m+1$ and 
$(-K_X\cdot l^{p, 0})\geq m+1$. Moreover, the integer $a$ do not 
depend on the choices of $z$ and $l$. 
Then $\sE':=\sE\otimes q^*\sO_{\pr^{m+1}}(-a)$ satisfies that 
$\sE'|_{(p')^{-1}(z)}\simeq\sO_{\pr^m}^{\oplus m+1}$ 
for any (closed) point $z\in\pr^{m+1}$. 
Thus $\sE_1:=p'_*\sE'$ is a locally free sheaf on $\pr^{m+1}$ of rank $m+1$ and 
$(p')^*\sE_1\simeq\sE'$ holds by the cohomology and base change theorem. 
Hence $\sE_1$ is isomorphic to one of the following 
\begin{enumerate}
\renewcommand{\theenumi}{\arabic{enumi}}
\renewcommand{\labelenumi}{\rm{(\theenumi)}}
\item\label{vb_tp_triv}
$\sO_{\pr^{m+1}}(b)^{\oplus m+1}$, 
\item\label{vb_tp_sato1}
$\sO_{\pr^{m+1}}(b)^{\oplus m}\oplus\sO_{\pr^{m+1}}(b+1)$ or 
\item\label{vb_tp_sato2}
$T_{\pr^{m+1}}\otimes\sO_{\pr^{m+1}}(b-1)$
\end{enumerate}
for some $b\in\Z$ by \eqref{vb_pp}. 
Take a line $l'$ in a fiber $(\simeq\pr^m)$ of $q$. 
Then we have $\sE|_{l'}\simeq\sO_{\pr^1}(a')^{\oplus m+1}$ for some $a'\in\Z$ 
by the same reason. 
Thus we have $(m+1)a'=(\det\sE\cdot l')=(\det\sE_1\cdot p'_*l')$. 
If $\sE_1$ is isomorphic to either of type \eqref{vb_tp_sato1} or \eqref{vb_tp_sato2}, 
then $(\det\sE_1\cdot p'_*l')=(m+1)b+1$. This leads to a contradiction. Hence 
$\sE_1\simeq\sO_{\pr^{m+1}}(b)^{\oplus m+1}$. 
In particular $X$ is isomorphic to $\pr_{\pr^{m+1}}(T_{\pr^{m+1}})\times\pr^m$. 
\end{proof}

\medskip

\section{Inductive step}\label{rays_section}

\smallskip

In this section, we prove Conjecture $\AGM^n_\rho$ under the conditions that 
Conjectures $\AGM^{n'}_{\rho-1}$ are true for small $n'$ and there exist special 
extremal rays for Fano manifolds satisfying the assumptions of 
Conjecture $\AGM^n_\rho$. 

We recall the result of Wi\'sniewski. 

\begin{thm}[{Wi\'sniewski's inequality \cite{wisn}}]\label{wisn_ineq}
Let $X$ be a smooth projective variety and $R\subset\overline{\NE}(X)$ be 
a $K_X$-negative extremal ray. Then any nontrivial fiber $F$ of $\phi_R$ 
$($the contraction morphism associated to $R$$)$ satisfies the inequality 
\[
\dim F\geq\dim X-\dim\Exc(\phi_R)+l(R)-1.
\]
\end{thm}

Together with the result of H\"oring and Novelli \cite{HN}, we get the following. 

\begin{thm}\label{bundle_criterion}
Let $X$ be a smooth projective variety and $R\subset\overline{\NE}(X)$ be 
a $K_X$-negative extremal ray. If any fiber $F$ of $\phi_R$ satisfies that 
$\dim F\leq l(R)-1$, then the morphism $\phi_R\colon X\rightarrow X_R$ is 
a $\pr^{l(R)-1}$-fibration. 
\end{thm}

\begin{proof}
For any nontrivial fiber $F$ of $\phi_R$, we have $\dim F=l(R)-1$ and 
$\dim\Exc(\phi_R)=\dim X$ by Theorem \ref{wisn_ineq}. 
Thus we can apply \cite[Theorem 1.3]{HN}.
\end{proof}

Using this, we get the key proposition in this section. 

\begin{proposition}\label{induction_prop}
Let $X$ be an $n$-dimensional Fano manifold of the pseudoindex $\iota$. 
Assume that there exists an extremal 
ray $R\subset\NE(X)$ such that any fiber $F$ of $\phi_R$ satisfies that 
$\dim F\leq\iota-1$. 
\begin{enumerate}
\renewcommand{\theenumi}{\arabic{enumi}}
\renewcommand{\labelenumi}{\rm{(\theenumi)}}
\item\label{induction_prop_gm}
If $X$ satisfies the assumptions of Conjecture $\GM^n_\rho$ for some fixed 
$\rho\geq 2$ and Conjecture $\GM^{n+1-\iota}_{\rho-1}$ is true, 
then $X$ is isomorphic to $(\pr^{\iota-1})^{\rho}$. 
\item\label{induction_prop_agm3}
If $X$ satisfies the assumptions of Conjecture $\AGM^n_\rho$ for some fixed 
$\rho\geq 2$ and Conjecture $\AGM^{n+1-\iota}_{\rho-1}$ is true, 
then $X$ is isomorphic to one of 
in the list of Conjecture $\AGM^n_\rho$. 
\end{enumerate}
\end{proposition}

\begin{proof}
The morphism $\phi_R\colon X\rightarrow X_R$ is a $\pr^{\iota-1}$-fibration by 
Theorem \ref{bundle_criterion}. We replace $X_R$ by $Y$ for simplicity. 
We know that $Y$ is an 
$(n+1-\iota)$-dimensional Fano manifold with $\rho_Y=\rho_X-1$ and 
$\iota_Y\geq\iota_X$ by Theorem \ref{pnbdle_lemma}. 

\eqref{induction_prop_gm}
We have the inequalities 
\[
\iota_Y\geq\iota\geq(n+\rho)/\rho\geq\left(n+1-\iota+(\rho-1)\right)/(\rho-1). 
\]
Thus $Y$ is isomorphic to $(\pr^{\iota-1})^{\rho-1}$ 
since we assume that Conjecture $\GM^{n+1-\iota}_{\rho-1}$ is true. 
Since $Y$ is rational, the morphism $\phi_R$ is a projective space bundle 
by Proposition \ref{brauer_prop}. 
Therefore $X$ is isomorphic to $(\pr^{\iota-1})^{\rho}$ by Corollary \ref{vb_gm}. 

\eqref{induction_prop_agm3}
We have the inequalities 
\[
\iota_Y\geq\iota\geq(n+\rho-1)/\rho\geq\left(n+1-\iota+(\rho-1)-1\right)/(\rho-1). 
\]
Thus $Y$ is isomorphic to one of $(\pr^{\iota-1})^{\rho-1}$, 
$\Q^\iota\times(\pr^{\iota-1})^{\rho-2}$, $\pr_{\pr^\iota}(\sO^{\oplus \iota-1}\oplus
\sO(1))\times(\pr^{\iota-1})^{\rho-3}$, $\pr_{\pr^\iota}(T_{\pr^\iota})
\times(\pr^{\iota-1})^{\rho-3}$ or 
$\pr^\iota\times(\pr^{\iota-1})^{\rho-2}$ 
since we assume that Conjecture $\AGM^{n+1-\iota}_{\rho-1}$ is true. 
Since $Y$ is rational, the morphism $\phi_R$ is a projective space bundle 
by Proposition \ref{brauer_prop}. 
Therefore $X$ is isomorphic to one of in the list of 
Conjecture $\AGM^n_\rho$ by Corollaries \ref{vb_gm} and \ref{vb_agm}. 
\end{proof}

\medskip

\section{Finding a special extremal ray}\label{domunsplit_section}

\smallskip

In this section, we show that Fano manifolds satisfying the assumptions 
in Conjecture $\AGM^n_\rho$ $(\rho\geq 2)$ have an extremal ray $R\subset\NE(X)$ 
such that any fiber $F$ of $\phi_R$ is of dimension $\leq\iota_X-1$ 
under the assumption that there exist numerically independent 
unsplit and dominating families of rational curves $V^1,\dots,V^{\rho-1}$ on $X$. 
This is a kind of generalization of Wi\'sniewski's result \cite[Lemma 4]{wisn91}. 

\begin{thm}\label{DUI_thm}
Let $X$ be an $n$-dimensional Fano manifold with $\rho:=\rho_X\geq 2$ 
which satisfies that $\iota_X\geq(n+\rho-1)/\rho$. 
Assume that there exist numerically independent 
unsplit and dominating families of rational curves $V^1,\dots,V^{\rho-1}$ on $X$. 
Then there exists an extremal ray $R\subset\NE(X)$ 
such that any fiber $F$ of $\phi_R$ is of dimension $\leq\iota_X-1$. 
\end{thm}

\begin{proof}
First, we prove the following: 

\begin{claim}\label{claim1}
For any extremal ray $R\subset\NE(X)$ with 
$R\not\subset\sum_{i=1}^{\rho-1}\R[V^i]$, 
the contraction morphism $\phi_R\colon X\rightarrow X_R$ is either 
\begin{enumerate}
\renewcommand{\theenumi}{\roman{enumi}}
\renewcommand{\labelenumi}{\rm{(\theenumi)}}
\item\label{claim11}
a divisorial contraction and any nontrivial fiber is of dimension $\iota_X$, 
or
\item\label{claim12}
of fiber type and any fiber is of dimension $\geq\iota_X-1$. 
\end{enumerate}
\end{claim}

\begin{proof}[Proof of Claim \ref{claim1}]
Take an arbitrary fiber $F$ of $\phi_R$. For any point $x\in F$, 
we have 
\[
\dim\Locus(V^1,\dots,V^{\rho-1})_x\geq\sum_{i=1}^{\rho-1}\left((-K_X
\cdot V^i)-1\right)\geq(\iota_X-1)(\rho-1)
\]
by Lemma \ref{no_l1}. Since $\NC(\Locus(V^1,\dots,V^{\rho-1})_x, X)
=\sum_{i=1}^{\rho-1}\R[V^i]$ (by Lemma \ref{no_l3} \eqref{no_l32}) 
and $\NC(F, X)=\R R$, we have 
$\dim(F\cap\Locus(V^1,\dots,V^{\rho-1})_x)=0$. 
Hence 
\[\dim F\leq n-\dim\Locus(V^1,\dots,V^{\rho-1})_x\leq n-(\iota_X-1)(\rho-1)
\leq\iota_X.
\]
Moreover, we have 
\[
\dim F\geq n-\dim\Exc(\phi_R)+l(R)-1\geq\iota_X-1
\]
by Theorem \ref{wisn_ineq}. 
Hence the assertion follows. 
\end{proof}

Next, we prove the following: 

\begin{claim}\label{claim2}
Take arbitrary distinct extremal rays $R$, $R'\subset\NE(X)$ with 
$R\not\subset\sum_{i=1}^{\rho-1}\R[V^i]$. 
Assume that any fiber $F'$ of $\phi_{R'}$ intersects some fiber $F$ of $\phi_R$. 
Then the morphism $\phi_{R'}$ also satisfies either the property \eqref{claim11} or 
\eqref{claim12} in Claim \ref{claim1}. Moreover, the following holds: 
\begin{enumerate}
\renewcommand{\theenumi}{\arabic{enumi}}
\renewcommand{\labelenumi}{\rm{(\theenumi)}}
\item\label{claim21}
If $\phi_R$ is a divisorial contraction, then $\phi_{R'}$ is of fiber type and any fiber 
of $\phi_{R'}$ is of dimension $\leq\iota_X-1$. 
\item\label{claim22}
If $\phi_{R'}$ is a divisorial contraction, then any fiber 
of $\phi_R$ that intersects some fiber of $\phi_{R'}$ 
is of dimension $\leq\iota_X-1$. 
\end{enumerate}
\end{claim}

\begin{proof}[Proof of Claim \ref{claim2}]
We can assume that $\NC(X)=\R R+\R R'+\sum_{i=1}^{\rho-2}\R[V^i]$ 
by renumbering $V^1,\dots,V^{\rho-1}$. 
Then we have $\NC(\Locus(V^1,\dots, V^{\rho-2})_F, X)=\R R
+\sum_{i=1}^{\rho-2}\R[V^i]$ by Lemma \ref{no_l3} \eqref{no_l32} and 
\begin{eqnarray*}
\dim\Locus(V^1,\dots, V^{\rho-2})_F & \geq & \dim F
+\sum_{i=1}^{\rho-2}\left((-K_X\cdot V^i)-1\right)\\
 & \geq & \dim F+(\rho-2)(\iota_X-1)\geq n-\iota_X
\end{eqnarray*}
holds by Lemma \ref{no_l1} and Claim \ref{claim1}. Moreover, 
if $\phi_R$ is a divisorial contraction, then we have 
$\dim\Locus(V^1,\dots, V^{\rho-2})_F
\geq n+1-\iota_X$ since $\dim F=\iota_X$. 
Since $\NC(F', X)=\R R'$, we have $\dim(F'\cap\Locus(V^1,\dots, V^{\rho-2})_F)=0$. 
Thus $\dim F'\leq n-\dim\Locus(V^1,\dots, V^{\rho-2})_F\leq\iota_X$. 
If $\phi_R$ is a divisorial contraction, then $\dim F'\leq\iota_X-1$. 
Moreover, $\dim F'\geq n-\dim\Exc(\phi_{R'})+l(R')-1\geq\iota_X-1$ holds 
by Theorem \ref{wisn_ineq}. If $\phi_{R'}$ is a divisorial contraction, then $\dim F'\geq\iota_X$. Therefore the assertion follows. 
\end{proof}

Assume that there exists an extremal ray $R\subset\NE(X)$ with 
$R\not\subset\sum_{i=1}^{\rho-1}\R[V^i]$ such that the contraction morphism $\phi_R$ 
is a divisorial contraction. Set $E:=\Exc(\phi_R)$. Then there exists an extremal ray 
$R'\subset\NE(X)$ with $R'\neq R$ such that $(E\cdot R')>0$ since $\NE(X)$ is
spanned by finite number of extremal rays. 
Then any fiber $F'$ of $\phi_{R'}$ intersects $E$. 
Thus $\dim F'\leq\iota_X-1$ by Claim \ref{claim2} \eqref{claim21}. 

Hence we can assume that any extremal ray $R\subset\NE(X)$ with 
$R\not\subset\sum_{i=1}^{\rho-1}\R[V^i]$ satisfies 
that the contraction morphism $\phi_R$ is of fiber type. 
We fix an extremal ray $R\subset\NE(X)$ with 
$R\not\subset\sum_{i=1}^{\rho-1}\R[V^i]$. 
Then any extremal ray $R'\subset\NE(X)$ with $R'\neq R$ 
satisfies either the property \eqref{claim11} or 
\eqref{claim12} in Claim \ref{claim1} by Claim \ref{claim2}. 

Assume that there exists an extremal ray $R'\subset\NE(X)$ with $R'\neq R$ 
such that the contraction morphism $\phi_{R'}$ is a divisorial contraction. 
Set $E':=\Exc(\phi_{R'})$. Then there exists an extremal ray 
$R''\subset\NE(X)$ with $R''\neq R'$ such that $(E'\cdot R'')>0$. 
If $R''\not\subset\sum_{i=1}^{\rho-1}\R[V^i]$, then 
any fiber of the morphism $\phi_{R''}$ has of dimension $\leq\iota_X-1$ 
by Claim \ref{claim2} \eqref{claim22}. Thus we can assume that 
$R''\subset\sum_{i=1}^{\rho-1}\R[V^i]$. In particular, $\rho$ must be bigger than or 
equal to three. 
We can assume that $\NC(X)=\R R+\R R'+\R R''+\sum_{i=1}^{\rho-3}\R[V^i]$ 
by renumbering $V^1,\dots,V^{\rho-1}$
since $R\not\subset\sum_{i=1}^{\rho-1}\R[V^i]$ and two distinct extremal rays 
$R'$ and $R''$ are in $\sum_{i=1}^{\rho-1}\R[V^i]$. 
Take any fiber $F''$ of $\phi_{R''}$. Then we can take a fiber $F'$ of $\phi_{R'}$ 
such that $F'\cap F''\neq\emptyset$ since $(E'\cdot R'')>0$ holds. 
Then $\NC(\phi_R^{-1}(\phi_R(F')), X)=\R R+\R R'$ and 
\begin{eqnarray*}
\dim\phi_R^{-1}(\phi_R(F'))\geq\iota_X-1+\dim\phi_R(F')=\iota_X-1+
\dim F'=2\iota_X-1
\end{eqnarray*}
since any fiber of $\phi_R$ is of dimension $\geq\iota_X-1$ and 
the restriction morphism 
$\phi_R|_{F'}\colon F'\rightarrow\phi_R(F')$ is a finite morphism. 
Moreover, we have 
\begin{eqnarray*}
\NC(\Locus(V^1,\dots,V^{\rho-3})_{\phi_R^{-1}(\phi_R(F'))}, X)
 & = & \R R+\R R'+\sum_{i=1}^{\rho-3}\R[V^i]\\
\dim\Locus(V^1,\dots,V^{\rho-3})_{\phi_R^{-1}(\phi_R(F'))} & \geq & 
\dim\phi_R^{-1}(\phi_R(F'))
+\sum_{i=1}^{\rho-3}\left((-K_X\cdot V^i)-1\right)\\
 & \geq & n+1-\iota_X
\end{eqnarray*}
by Lemmas \ref{no_l1} and \ref{no_l3} \eqref{no_l32}. 
Thus $\dim(F''\cap\Locus(V^1,\dots,V^{\rho-3})_{\phi_R^{-1}(\phi_R(F'))})=0$. 
Therefore $\dim F''\leq n-\dim\Locus(V^1,\dots,V^{\rho-3})_{\phi_R^{-1}(\phi_R(F'))}
\leq\iota_X-1$ for any fiber $F''$ of $\phi_{R''}$. 

Hence we can assume that any extremal ray $R_1\subset\NE(X)$ satisfies that 
the contraction morphism $\phi_{R_1}$ 
is of fiber type. For any fiber $F_1$ of $\phi_{R_1}$, we have 
$\dim F_1\geq\iota_X-1$ by Theorem \ref{wisn_ineq}. 
We can assume that there exists an extremal ray $R_1\subset\NE(X)$ and a fiber $F_1$ 
of $\phi_{R_1}$ such that the dimension of $F_1$ is bigger than or equal to $\iota_X$. 
Take any $(\rho-1)$-dimensional extremal face $S\subset\NE(X)$ such that 
$R_1\subset S$ and let $\phi_S\colon X\rightarrow X_S$ be the contraction morphism 
of $S$. 
Then there exists a fiber $F_S$ of $\phi_S$ such that $\dim F_S\geq n+1-\iota_X$. 
Indeed, let $x_S\in X_S$ be the image of $F_1\subset X$. Then 
$\dim\phi_S^{-1}(x_S)\geq\iota_X+(\rho-2)(\iota_X-1)\geq n+1-\iota_X$. 
We also take an extremal ray $R_0\subset\NE(X)$ that $R_0\cap S=0$. Then for 
any fiber $F_0$ of $\phi_{R_0}$, we have $\dim(F_0\cap F_S)=0$. Therefore 
$\dim F_0\leq n-\dim F_S\leq\iota_X-1$ holds. 

Consequently, we complete the proof of Theorem \ref{DUI_thm}. 
\end{proof}

As a corollary, we get the following result. 

\begin{corollary}\label{rccor}
Let $X$ be an $n$-dimensional Fano manifold satisfying the assumptions 
of Conjecture $\AGM^n_\rho$ for some $\rho\geq 2$. Let $V^1,\dots,V^k$ be 
families of rational curves on $X$ as in Construction \ref{const}. If $V^i$ are unsplit 
for all $1\leq i\leq k$, then there exists an extremal ray $R\subset\NE(X)$ 
such that any fiber $F$ of $\phi_R$ is of dimension $\leq\iota_X-1$. 
\end{corollary}

\begin{proof}
We know that $k=\rho_X$ by Proposition \ref{no_p3}. 

If $k\geq\rho+1$, then 
$n\geq\sum_{i=1}^k\left(\dim X-\dim\Locus(V^i)+(-K_X\cdot V^i)-1\right)\geq
k(\iota_X-1)\geq(\rho+1)(\iota_X-1)$ 
by Lemma \ref{no_l4}. We note that $\iota_X\geq 2$ and $\iota_X\rho+1-\rho\geq n$. 
Thus $\iota_X=2$, $n=k=\rho+1$ and $V^1,\dots,V^n$ are numerically independent 
dominating and unsplit family of rational curves such that $(-K_X\cdot V^i)=2$ 
for any $1\leq i\leq n$. Therefore $X$ is isomorphic to $(\pr^1)^n$ by 
\cite[Theorem 1.1]{Occ}. 

Hence we can assume that $\rho_X=k=\rho$. Then 
\[
n\geq\sum_{i=1}^k\left(\dim X-\dim\Locus(V^i)+(-K_X\cdot V^i)-1\right)\geq
\rho(\iota_X-1)\geq n-1
\]
by Lemma \ref{no_l4}. Thus at least $\rho-1$ number of families in 
$\{V^1,\dots,V^\rho\}$ are dominating families of rational curves. 
Therefore we can apply Theorem \ref{DUI_thm}. 
\end{proof}

\medskip

\section{Proof of Proposition \ref{mainAGMprop}}\label{five_section}

\smallskip

In this section, we prove Proposition \ref{mainAGMprop}. First, we consider 
Proposition \ref{mainAGMprop} \eqref{mainAGMprop1}. 

\begin{proposition}\label{AGMn2}
Let $X$ be an $n$-dimensional Fano manifold with $\rho_X\geq 2$ and 
$\iota_X\geq(n+1)/2$. Then there exists an extremal ray $R\subset\NE(X)$ such that 
any fiber $F$ of $\phi_R$ satisfies that $\dim F\leq \iota_X-1$. 
\end{proposition}

\begin{proof}
Take families $V^1,\dots,V^k$ 
of rational curves on $X$ as in Construction \ref{const}. 
It is enough to show that all of $V^1,\dots,V^k$ are unsplit by Corollary \ref{rccor}. 
Assume that there exists a non-unsplit family, say $V^j$. Then 
$(-K_X\cdot V^j)\geq 2\iota_X$. Thus $j=k=1$, $(-K_X\cdot V^1)=2\iota_X$ 
and $n=2\iota_X-1$ by Lemma \ref{no_l4}. However, since $3\iota_X>(-K_X\cdot V^1)
>n+1-\iota_X=\iota_X$, we have $\rho_X=1$ by Theorem \ref{no_t2} \eqref{no_t22}. 
This leads to a contradiction. Therefore all of $V^1,\dots,V^k$ are unsplit families. 
\end{proof}

\begin{proposition}\label{AGMn3}
Let $X$ be an $n$-dimensional Fano manifold with $\rho_X\geq 3$ and 
$\iota_X\geq(n+2)/3$. 
Then there exists an extremal ray $R\subset\NE(X)$ such that 
any fiber $F$ of $\phi_R$ satisfies that $\dim F\leq \iota_X-1$. 
\end{proposition}

\begin{proof}
Take families $V^1,\dots,V^k$ 
of rational curves on $X$ as in Construction \ref{const}. 
It is enough to show that all of $V^1,\dots,V^k$ are unsplit by Corollary \ref{rccor}. 

Assume that there exists a non-unsplit family, say $V^j$. Then 
$(-K_X\cdot V^j)\geq 2\iota_X$. 
If $k=1$, then we have $3\iota_X>3\iota_X-1\geq n+1\geq(-K_X\cdot V^1)\geq 
2\iota_X\geq n+2-\iota_X>n+1-\iota_X$. Thus $\rho_X=1$ 
by Theorem \ref{no_t2} \eqref{no_t22}. 
This leads to a contradiction. Hence $k\geq 2$. 
By Lemma \ref{no_l4}, we have 
\begin{eqnarray*}
n & \geq & \sum_{i=1}^k\left(n-\dim\Locus(V^i)+(-K_X\cdot V^i)-1\right)\\
 & \geq & (2\iota_X-1)+(k-1)(\iota_X-1)\geq 3\iota_X-2\geq n.
\end{eqnarray*}
Hence we get $k=2$, $n=3\iota_X-2$, both $V^1$ and $V^2$ are dominating families, 
$(-K_X\cdot V^j)=2\iota_X$ and $(-K_X\cdot V^i)=\iota_X$ holds, where 
$\{i$, $j\}=\{1$, $2\}$. 
Thus for general $x\in X$, we have 
\begin{eqnarray*}
\dim\Locus(V^j, V^i)_x & \geq & (-K_X\cdot V^j)-1+(-K_X\cdot V^i)-1=n,\\
\NC(\Locus(V^j, V^i)_x, X) & = & \R[V^j]+\R[V^i]
\end{eqnarray*}
by Lemmas \ref{no_l1} and \ref{no_l3} \eqref{no_l32}. 
Thus $\rho_X=2$. This leads to a contradiction. 
Therefore all of $V^1,\dots,V^k$ are unsplit families. 
\end{proof}

By Propositions \ref{AGMn2}, \ref{AGMn3} and \ref{induction_prop} 
\eqref{induction_prop_agm3}, we have proved Proposition \ref{mainAGMprop}
\eqref{mainAGMprop1}. 

Next, we consider Proposition \ref{mainAGMprop} \eqref{mainAGMprop2}. 
To do this, it is enough to study five-dimensional Fano manifolds $X$ 
with $\iota_X=2$ and 
$\rho_X=4$ by \cite{isk1, isk2, sho, MoMu, ACO, NO}, 
Propositions \ref{AGMn2} and \ref{AGMn3}. 

\begin{proposition}\label{five}
Let $X$ be a five-dimensional Fano manifold with $\iota_X=2$ and $\rho_X=4$. 
Then $X$ is isomorphic to one of $\pr_{\pr^2}(\sO\oplus\sO(1))\times(\pr^1)^2$, 
$\pr_{\pr^2}(T_{\pr^2})\times(\pr^1)^2$ or $\pr^2\times(\pr^1)^3$. 
\end{proposition}

\begin{proof}
Take families $V^1,\dots,V^k$ 
of rational curves on $X$ as in Construction \ref{const}. 
We note that Conjecture $\AGM^4_3$ is true by \cite{MoMu} and 
Proposition \ref{mainAGMprop} \eqref{mainAGMprop1}. 
Thus it is enough to show that all of $V^1,\dots,V^k$ are unsplit by 
Corollary \ref{rccor} 
and Proposition \ref{induction_prop} \eqref{induction_prop_agm3}. 

Assume that $V^j$ is non-unsplit for some $1\leq j\leq k$. 
Such $V^j$ is unique and $k\leq 3$ holds by the inequalities 
\[
5\geq\sum_{i=1}^k\left(5-\dim\Locus(V^i)+(-K_X\cdot V^i)-1\right)
\geq (2\cdot 2-1)+(k-1)(2-1)
\]
in Lemma \ref{no_l4}. 
Moreover, we know that $j=1$ by Lemma \ref{novlem}. 

Assume that $k=3$. Then we have $(-K_X\cdot V^1)=4$, $(-K_X\cdot V^i)=2$ 
and $V^i$ is a dominating family for $i=2$, $3$ by the inequalities 
\[
5\geq\sum_{i=1}^3\left( 
5-\dim\Locus(V^i)+(-K_X\cdot V^i)-1\right)\geq (2\cdot 2-1)+2(2-1)=5
\]
 in Lemma \ref{novlem}. This leads to a contradiction since $V^1$ is 
a minimal dominating family. Thus $k\leq 2$. 

Assume that $k=2$. We repeat the proof in \cite[Theorem 5]{NO}. 
We have either 
$\dim\Locus(V^2)=4$ and $(-K_X\cdot V^1)=4$ and $(-K_X\cdot V^2)=2$, or $\dim\Locus(V^2)=5$ and $(-K_X\cdot V^1)\geq 4>3\geq(-K_X\cdot V^2)$ 
by the inequalities 
\[
5\geq\sum_{i=1}^2\left( 
5-\dim\Locus(V^i)+(-K_X\cdot V^i)-1\right)\geq (2\cdot 2-1)+(2-1)=4
\] 
in Lemma \ref{no_l4}. If $\dim\Locus(V^2)=5$, then this leads to a contradiction 
since $V^1$ is a minimal dominating family. We can assume that 
$\dim\Locus(V^2)=4$. 
For a general $x\in X$, we have $\Locus(V^1_x)=(\pi^1)^{-1}(\pi^1(x))$ 
by Lemma \ref{no_l4}, where $\pi^1\colon X\dashrightarrow Z^1$ is the 
$\rc(\sV^1)$-fibration. 
Thus $\Locus(V^1_x)\cap\Locus(V^2)\neq\emptyset$ 
since $V^2$ is a horizontal dominating family with respect to $\pi^1$. 
Hence we have 
\begin{eqnarray*}
\dim\Locus(V^1,V^2)_x & \geq & 4\\
\NC(\Locus(V^1,V^2)_x, X) & = & \R[V^1]+\R[V^2]
\end{eqnarray*}
by Lemmas \ref{no_l1} and \ref{no_l3} \eqref{no_l32}. 
Therefore $\rho_X\leq 3$ by \cite[Theorem 1.2]{casagrande}. 
This leads to a contradiction. 

Assume that $k=1$. If $\dim\Locus(V^1_x)\geq 4$ 
for a general $x\in X$, then $\rho_X\leq 2$ 
by \cite[Theorem 1.2]{casagrande}. Hence $\dim\Locus(V^1_x)\leq 3$ for 
a general $x\in X$. 
Then $(-K_X\cdot V^1)=4$ by Proposition \ref{no_p1}. 
We have $3\iota_X=6>4=(-K_X\cdot V^1)=\dim X+1-\iota_X$. 
Thus $\rho_X\leq 3$ by Theorem \ref{no_t2} \eqref{no_t23}. 
This leads to a contradiction. 

Therefore we have proved Proposition \ref{five}. 
\end{proof}

As a consequence, we have proved Proposition \ref{mainAGMprop} 
\eqref{mainAGMprop2}.

\medskip

\section{Proof of Theorem \ref{mainthm}}\label{AM_section}

\smallskip

In this section, we prove Theorem \ref{mainthm}. By \cite{KO}, 
\cite[Theorem B]{wisn90}, \cite[Theorem]{wisn91}, \cite[Theorem 3]{NO}, 
\cite[Theorem 5.1]{novelli} and 
Proposition \ref{mainAGMprop} \eqref{mainAGMprop2}, it is enough to show the 
following. 

\begin{thm}\label{AM3_thm}
Set $r\geq 3$. If $X$ is a $(3r-2)$-dimensional Fano manifold with $r_X=r$ and 
$\rho_X=3$, then $X$ is isomorphic to one of $\Q^r\times(\pr^{r-1})^2$, 
$\pr_{\pr^r}(\sO^{\oplus r-1}\oplus\sO(1))\times\pr^{r-1}$ or 
$\pr_{\pr^r}(T_{\pr^r})\times\pr^{r-1}$. 
\end{thm}

\begin{proof}
By \cite[Theorem 3]{NO}, we have $\iota_X=r_X=r$. 
By Proposition \ref{AGMn3} and Theorem \ref{bundle_criterion}, there exists 
an extremal ray $R\subset\NE(X)$ such that the associated contraction morphism 
$\phi_R\colon X\to Y$ is a $\pr^{r-1}$-fibration. The variety $Y$ is a 
$(2r-1)$-dimensional Fano manifold with $\iota_Y\geq r$ and $\rho_Y=2$ by Lemma 
\ref{pnbdle_lemma}. By \cite[Theorem 3]{NO}, 
we have $\iota_Y=r$. By Proposition \ref{AGMn2} 
and Theorem \ref{bundle_criterion}, there exists 
an extremal ray $S\subset\NE(Y)$ such that the associated contraction morphism 
$\phi_S\colon Y\to Z$ is a $\pr^{r-1}$-fibration.

\begin{claim}\label{AM3_claim}
The variety $Z$ is isomorphic to either $\pr^r$ or $\Q^r$. 
\end{claim}

\begin{proof}[Proof of Claim \ref{AM3_claim}]
Set $\pi:=\phi_S\circ\phi_R\colon X\to Z$. Let $R'\subset\NE(X)$ be the extremal ray 
such that the morphism $\pi$ corresponds to the extremal face $R+R'\subset\NE(X)$. 
Choose any extremal ray $R''\subset\NE(X)$ with $R''\neq R$, $R'$. 
Then any nontrivial fiber $F$ of $\phi_{R''}\colon X\to X_{R''}$ satisfies that 
$\dim F\leq r$ since $\pi|_F\colon F\to Z$ is a finite morphism.  On the other hand, 
by Theorem \ref{wisn_ineq}, 
\[
\dim F\geq \dim X-\dim\Exc(\phi_{R''})+l(R'')-1. 
\]
Thus $l(R'')=r$ and there are three possibilities: 
\begin{enumerate}
\renewcommand{\theenumi}{\arabic{enumi}}
\renewcommand{\labelenumi}{\rm{(\theenumi)}}
\item\label{AM3_claim1}
$\phi_{R''}$ is a divisorial contraction and any fiber $F$ of $\phi_{R''}$ satisfies that 
$\dim F=r$. 
\item\label{AM3_claim2}
$\phi_{R''}$ is of fiber type and any fiber $F$ of $\phi_{R''}$ satisfies that 
$\dim F=r$. 
\item\label{AM3_claim3}
$\phi_{R''}$ is of fiber type and a general fiber $F$ of $\phi_{R''}$ satisfies that 
$\dim F=r-1$. 
\end{enumerate}

We consider the case \eqref{AM3_claim1}. Then $F\simeq\pr^r$ by 
\cite[Theorem 4.1 (iii)]{AW93}. Thus $Z\simeq\pr^r$ by \cite[Theorem 1]{OW}. 
We consider the case \eqref{AM3_claim2}. Then a general fiber $F$ is isomorphic to 
$\Q^r$ by \cite{KO}. Thus $Z\simeq\pr^r$ or $\Q^r$ by \cite{cho-sato}. 
We consider the case \eqref{AM3_claim3}. Set 
\[
B:=\{x\in X_{R''}\,\,|\,\,\dim\phi_{R''}^{-1}(x)\geq r\}.
\]
Since $\codim_X(\phi^{-1}_{R''}(B))\geq 2$, we can take a general (complete) 
very free rational 
curve $C$ on $X\setminus\phi_{R''}^{-1}(B)$ such that $C'':=\phi_{R''}(C)$ is not 
a point by \cite[Proposition II.3.7, Theorems IV.3.10 and V.2.13]{kollar}. 
By \cite[Theorem 1.3]{HN}, 
$\phi_{R''}^{-1}(x)$ is scheme-theoretically isomorphic to $\pr^{r-1}$ for any $x\in C''$. 
Let $\nu\colon\pr^1\to C''\hookrightarrow X_{R''}$ be the normalization morphism 
and set $T:=X\times_{X_{R''}}\pr^1$ as in 
Definition \ref{hor_section}. Since $T\to \pr^1$ is a $\pr^{r-1}$-fibration, 
$T$ is a toric variety. For any fiber $F''$ of $T\to \pr^1$, the morphism 
$\pi\colon X\to Z$ restricted to the image of $F''$ is a finite morphism. Since $C$ is 
general, the morphism $T\to Z$ is surjective. Therefore $Z\simeq\pr^r$ by 
\cite[Theorem 1]{OW}. 
\end{proof}

By using Proposition \ref{brauer_prop} and Corollary \ref{vb_agm} twice of each, 
we get the possibilities of the structures of $Y$ and $X$. Thus we get the assertion. 
\end{proof}

As a consequence, we have complete the proof of Theorem \ref{mainthm}.

\medskip

\noindent K.\ Fujita

Department of Mathematics, Graduate School of Science, 
Kyoto University, Oiwake-cho, 
Kitashirakawa, Sakyo-ku, Kyoto 606-8502, Japan 

fujita@math.kyoto-u.ac.jp
\end{document}